\documentclass{article}
\usepackage[utf8]{inputenc}
\usepackage{amsmath,amssymb, amsthm}
\usepackage{booktabs}
\usepackage{color}
\usepackage{geometry}
\usepackage{comment}

\def\psiV{\psi^{(V)}}
\def\Eet{E^{( \eta, \theta ) }}
\def\Ett{E^{( \theta, \theta ) }}
\def\ee{\mathbb{E}}
\def\pp{\mathbb{P}}
\def\rr{\mathbb{R}}
\def\zz{\mathbb{Z}}
\def\F{\mathcal{F}}
\def\e{\mathrm{e}}
\def\d{\mathrm{d}}
\def\del{\partial}
\def\1{\boldsymbol{1}}
\def\O{\mathcal{O}}
\def\mfe{\mathfrak{e}}
\def\Var{\mathrm{Var}}
\def\Cov{\mathrm{Cov}}
\def\mff{\mathfrak{f}}
\def\eps{\varepsilon}

\newtheorem{ccounter}{ccounter}[section]
\newtheorem{thm}[ccounter]{Theorem}
\newtheorem{lem}[ccounter]{Lemma}
\newtheorem{cor}[ccounter]{Corollary}
\newtheorem{defn}[ccounter]{Definition}
\newtheorem{prop}[ccounter]{Proposition}
\newtheorem{ass}[ccounter]{Assumption}
\newtheorem{ex}[ccounter]{Example}

\def\bet{\begin{thm}}
\def\eet{\end{thm}}
\def\bel{\begin{lem}}
\def\eel{\end{lem}}
\def\bas{\begin{ass}}
\def\eas{\end{ass}}
\def\bec{\begin{cor}}
\def\eec{\end{cor}}
\def\bed{\begin{defn}}
\def\eed{\end{defn}}
\def\bep{\begin{prop}}
\def\eep{\end{prop}}
\def\beq{\begin{equation}}
\def\eeq{\end{equation}}
\def\proof{\noindent {\bf Proof.}\ \ }
\def\bea{\begin{equation*}}
\def\eea{\end{equation*}}

\def\bex{\begin{ex}}
\def\eex{\end{ex}}
\def\remark{\noindent{\bf Remark. }}

\newtheorem{theorem}[ccounter]{Theorem}

\title{Upper tail bounds for stationary KPZ models}
\author{Benjamin Landon$^1$ \quad  Philippe Sosoe$^2$}
\date{\today}

\begin{document}

\maketitle

\begin{abstract}
    We present a proof of an upper tail bound of the correct order (up to a constant factor in the exponent) in two classes of stationary models in the KPZ universality class.
    
    The proof is based on an exponential identity due to Rains in the case of Last Passage Percolation with exponential weights, and recently re-derived by Emrah-Janjigian-Sepp\"ail\"ainen (EJS). Our proof follows very similar lines for the two classes of models we consider, using only general monotonocity and convexity properties, and can thus be expected to apply to many other stationary models.
\end{abstract}

\section{Introduction}
\let\thefootnote\relax\footnote{1. University of Toronto, \texttt{blandon@math.toronto.edu}. 2. Cornell University, \texttt{psosoe@math.cornell.edu}}In this paper, we derive estimates of the correct order, up to constant factors, for the upper tail of the distribution of the partition function in stationary integrable polymer models. Through an analogous argument, we also obtain the same result for the upper tail of a height function defined in terms of a model of Brownian motions interacting through a potential, at equilibrium. We had previously studied the fluctuations of the latter model at the level of the second moment with C. Noack in \cite{LNS}. For a specific choice of interaction potential, the height function in this model coincides in distribution with the log-partition function of the O'Connell-Yor semi-discrete polymer, but in general the model is not expected to be integrable.

Our starting point is an identity for the generating function of an off-equilibrium partition function, evaluated at a certain point. In the context of last passage percolation with exponential weights, this identity was discovered by Rains \cite{R}. It was recently re-introduced into the study of stationary models by Emrah, Janjigian and Sepp\"al\"ainen \cite{EJS}, who also gave a very simple proof that is readily adapted to other stationary models, including the ones we study here. We refer to this identity as the Rains-EJS identity. It has already found applications to optimal order bounds on the tails and central moments in exponential last passage percolation (the ``zero temperature limit'' of polymers) \cite{EGO,EJS,EJS2}.

Our contribution is to show how the Rains-EJS identity, combined with simple monotonicity properties, leads very quickly to fluctuation bounds of optimal order, even in a model where there is a priori no obvious path interpretation of the height function or second class particles, and thus the \emph{transversal fluctuation} exponent central to geometric arguments in the study of models in the Kardar-Parisi-Zhang class cannot be directly meaningfully defined.

\subsection{Main Results}
Our first result concerns integrable models of stationary directed polymers with boundaries. Four families of such models, parametrized by pairs of numbers $(\theta,\mu)$ are known: the log-gamma polymer \cite{S}, the strict-weak polymer \cite{CSS,OO}, the beta random walk \cite{BC} and the inverse beta polymer \cite{TL}. These were unified in a common framework in \cite{CN}. We recall this framework in Section \ref{sec: partition}, and refer to that section for precise definitions. In particular, Table \ref{table: integrable} summarizes the edge weight densities and stationary parameter choices for each of the four families of models.

The polymer partition function is a sum over up-right paths from $(0, 0)$ to $(m,n)$ in $\zz^2$, weighted by a random environment. The specification of the environment distribution is what leads to the four integrable models. For definiteness we assume that $m$ and $n$ depend on a single asymptotic parameter $N $, and moreover that $m, n \to \infty$ in such a way that a certain characteristic direction condition (given below as \eqref{eqn: char-direction}) is satisfied. This latter condition is the necessary condition for the leading order asymptotic flucutations of the partition function to be of KPZ type. Theorem \ref{thm:disc-main} follows directly from Theorem \ref{thm:discrete-main-tech} and Proposition \ref{prop:a-tail} appearing later in the paper.

\begin{theorem} \label{thm:disc-main}
Let $Z_{m,n}(\mu,\theta)$  denote the partition function for one of the four integrable polymer models with boundaries: log-gamma polymer, strict-weak polymer, beta random walk and inverse beta polymer. See equation \eqref{eqn: Z} for a definition. Under the characteristic direction condition \eqref{eqn: char-direction}, there are constants $C(\theta,\mu)$ and $c(\theta,\mu)$ uniformly bounded above and below in compact intervals such that  for all $0<t\le c ( \theta, \mu) N^{2/3}$, we have
\[ c( \theta, \mu ) \e^{-C(\theta,\mu)t^\frac{3}{2}} \le \mathbb{P}(\log Z_{m,n} ( \theta, \mu) \ge \mathbb{E}[\log Z_{m,n} ( \theta, \mu)]+tN^{1/3})\le C(\theta, \mu) \e^{-c(\theta,\mu)t^\frac{3}{2}}.\]
\end{theorem}
We remark that under the same assumptions one can find an upper bound of the same order for the lower tail: see the remark following Theorem \ref{thm:discrete-main-tech}. Such an estimate is sub-optimal as one expects a tail like $\e^{ - c t^3}$, but it does allow one to remove the $\epsilon$ loss in the bound of order $N^{\frac{p}{3}+\epsilon}$ for the central moments of order $p$ derived by Noack and the second author in \cite{NS}.

Our second result concerns the height function in a family of diffusion models studied in \cite{LNS}. To define it, we consider the following system for $N$ diffusions $ \{ u_i (t) \}_{i=1}^N \in \rr^N$,
\begin{align} \label{eqn:sys}
\d u_1 &= - V' (u_1 ) \d t  + \d B_0 - \theta \d t + \d B_1 \notag \\
\d u_j &= - V' (u_j ) \d t  + V' (u_{j-1} ) \d t + \d B_j - \d B_{j-1} , \qquad j \geq 2 .
\end{align}
Here $\theta>0$ is a parameter.
Above, the $B_j (t)$'s are all independent Brownian motions. 
We will take the potential $V : \rr \to \rr$ to be smooth. Our full assumptions (under which the above diffusions are well-defined, see \cite{LNS}) are given below in Definition \ref{def:V-def} in Section \ref{sec: diffusion}. As we also recall in Section \ref{sec: diffusion}, the system \eqref{eqn:sys} has the unique stationary distribution $\omega_\theta$  \eqref{eqn: stationary}, which is the product measure with each factor proportional to $\e^{ - V (u) - \theta u}$.

The \emph{height function} is defined by
\[W_{N,t}^\theta=\sum_{j=1}^N u_j(t)-B_0(t)+\theta t.\]
When $\{u_j(0)\}_{j=1}^N$ are in equilibrium and $V(x)=\e^{-x}$, the distribution of $W_{N,t}^\theta$ coincides with the logarithm of the partition function of the stationary O'Connell-Yor polymer, introduced in \cite{OY}. In this case, O'Connell \cite{Oquantum} has introduced a triangular array of stochastic differential equations (SDEs) for closely related objects. This system contains equations identical in form to those satisfied by $W_{j,t}^\theta$, $2\le j\le N$, as well as \emph{different} diffusion equations for the distribution of free energies of ratios of the partition functions of multipath versions of the polymer.

For general $V$, in contrast to the O'Connell-Yor case, one does not expect an integrable structure beyond the existence of the stationary measure $\omega_\theta$. The next result shows that $W_{N,t}^\theta$ nevertheless exhibits upper tail moderate deviations consistent with those of the Tracy-Widom and Baik-Rains distributions, a hallmark of the KPZ class. In addition to a characteristic direction condition (similar to the discrete polymer case), this result requires that the point $\theta$ satisfy a non-vanishing curvature condition (given as \eqref{eqn:curvature} below; it depends only on the third derivative of the Laplace transform of the measure $\e^{ - V(u)} \d u$). It is expected that this is a necessary condition for the model to lie in the KPZ universality class. The result is a consequence of Theorem \ref{thm:main-diff} below. 
\begin{theorem}
Suppose $N$, $t$ and $\theta$ satisfy the characteristic direction condition \eqref{eqn:main-diff-1} and the curvature condition \eqref{eqn:curvature}. Then there are constants $c, C$ uniformly bounded for $\theta$ in compact sets of $(0,\infty)$, such that:
\[
C^{-1} \e^{ - C s^{3/2} } \leq \pp\left[ W^\theta_{N, t} - \ee[ W^\theta_{N, t} ]  > s N^{1/3} \right] \leq c^{-1} \e^{ - c s^{3/2} }
\]
whenever $0<s\le N^{2/3}$.
\end{theorem}
Note that this result is new even in the case $V(x)=\e^{-x}$. For this O'Connell-Yor case, Borodin-Corwin-Ferrari \cite{BCF} have shown that for (singular) initial data of so-called ``narrow wedge" type, the rescaled and centered free energy converges to the Tracy-Widom distribution $F_2$.  Semi-discrete polymers with log-gamma boundary sources were studied by Borodin-Corwin-Vet\"o \cite{BCFV} and the specific stationary case $\theta >0$ was studied by Imamura and Sasamoto \cite{IS}, where a different KPZ distribution, the Baik-Rains distribution, appears. 
Our result shows that, at least at the level of moderate deviations, the behavior of $W_{N,t}^\theta$ is stable under perturbations of the potential. 

In previous work with Noack \cite{LNS}, we had already shown that the $N^{\frac{1}{3}}$ scaling obtained by Sepp\"al\"ainen-Valko \cite{SV} (see also Moreno-Flores-Sepp\"al\"ainen-Valko \cite{MFSV}) in the O'Connell-Yor case persists also in the setting of general $V$. The conjecture that the diffusion models \eqref{eqn:sys} lie in the KPZ class for generic $V$ appears in the introduction to the monograph \cite{FSW}. One can view the models \eqref{eqn:sys} for general $V$ as (formal) scaling limits of certain discrete state-space interacting particle systems \cite{B,C}. In the case $V(x)=\e^{-x}$, this is a theorem, since the partition function of the O'Connell-Yor polymer appears as a limit of a $q$-TASEP process \cite{BC}. Although the convergence is not strong enough to directly control the moments or tails, this provides further heuristic evidence that the models \eqref{eqn:sys} belong to the KPZ class.

Recent work of Corwin and Ghosal \cite{CG} shows that for the KPZ equation, the tail behavior is ``more universal" than the asymptotic distribution itself, in the sense that the distribution can change under change of initial data, but the tail behavior remains stable.

Finally, we remark that tail estimates for the non-stationary models can be deduced from the results for the stationary ones. For example, an upper bound for the upper tail can be directly deduced from monotonicity and the Rains-EJS identity, as pointed out in \cite{EJS} for exponential last passage percolation. We give this short calculation in the O'Connell-Yor polymer case in Corollary \ref{cor:oy-non-stat}.

\subsection{Previous results}
There is a considerable literature on tail bounds for models in the KPZ class. We mention some previous results directly related to the models we study here, and direct the reader to the introduction of the recent work \cite{CH} by Corwin and Hegde for a more detailed discussion of recent developments.

In addition to the works already mentioned previously, for the (non-stationary) log-Gamma polymer, Barraquand, Corwin and Dimitrov \cite{BCD} obtain an upper bound for the upper tail in the moderate deviation regime which essentially matches the result in Theorem \ref{thm:disc-main}. Previously, Sepp\"al\"ainen and Georgiou \cite{SG} had obtained upper tail large deviations for this polymer model. Their result, like ours, uses a more probabilistic approach, relying on the stationary structure (the ``Burke property") uncovered in \cite{S}.

The method in \cite{BCD} is based on a formula from \cite{BCFV}, representing the moment generating function of the free energy as a Fredholm determinant with a kernel given by a complex but explicit contour integral involving a product of ratios of Gamma functions. The approach through Fredholm determinants has the advantage of being able to access the asymptotic distribution of the centered and rescaled free energy. On the other hand, since the Fredholm determinant is an infinite sum, it is not well suited in regimes where there is cancellation between the terms. In this case, one needs to use alternative representations for the moment generating function, such as the multiplicative identity exploited in \cite{CG} to obtain estimates for the lower tail of the solution of the KPZ equation. An approach to computing the moment generating function via the Riemann-Hilbert method was presented in \cite{CC}.

The method presented here not only gives upper and lower bounds, but also an upper bound for the lower tail of the free energy on the $N^{1/3}$ scale, albeit with a suboptimal $s^{3/2}$ exponent. See our recent paper  \cite{LS}, where this initial bound is used as an input to derive bounds with the correct $s^3$ exponent for the lower tail of the free energy of the O'Connell-Yor polymer. In this vein, we again mention the recent work \cite{CH} where an upper bound for the lower tail is obtained for the more intricate $q$-pushTASEP model.

After posting our result, we were made aware of new results of Janjigian, Emrah and Xie which overlap with our first Theorem. These authors obtain upper tail moderate deviations for three of the four integrable polymer models covered by Theorem \ref{thm:disc-main}. The results form part of Xie's thesis \cite{X} and will be reported in an upcoming paper by the authors. We note that these authors also obtain explicit constants in front of the $s^{3/2}$ tail exponent.

For last passage percolation, several moderate and large deviation results are available. Before the results \cite{EGO,EJS,EJS2} cited above, which rely on a probabilistic coupling approach as we do, several works deal with tail bounds for the passage time in the moderate deviation regime, and obtain results using integrable probability: \cite{BDMMZ,J,S2}. See also \cite{BG, BGK} for recent, refined large deviation results for the passage time, or equivalently, the Laguerre Unitary Ensemble. 

We also mention the papers \cite{CG,CG2,DT,Ts}, which deal with tails bounds for the continuum KPZ equation. We stress that this list is necessarily non-exhaustive, given the intense interest the question of tail bounds for these models has generated.

\subsection{Organization of paper}

The rest of the paper is organized into two larger parts: the first dealing with the discrete polymer models in Sections \ref{sec:disc-main}--\ref{sec:disc-tail}; and the second handling the diffusion system in Sections \ref{sec: diffusion}--\ref{sec:diff-tail}.

In Section \ref{sec:disc-main} we formally introduce the discrete polymer models via the Mellin transform framework, as well as state and prove the Rains-EJS identity for this model in Theorem \ref{thm: EJS}. In Section \ref{sec:disc-derivs} we derive some identities and properties of the derivatives of the log partition function wrt the boundary parameters. In Section \ref{sec:disc-exit} we derive an annealed tail bound for the exit point of the polymer path from the horizontal boundary (the exit point from the vertical boundary can be treated by identical arguments). Finally, our main results on tail estimates for the discrete polymer case are derived in Section \ref{sec:disc-tail}.

In Section \ref{sec: diffusion} we formally define the systems of diffusions we study as well as the assumptions on the potential $V$. In Section \ref{sec:diff-prelim} we state some preliminaries for the diffusions that we will require for our proofs. For example, we introduce a two-parameter model and state monotonicity properties wrt the boundary parameters as well as the Rains-EJS identity for this model. We also recall the definition of the pseudo-Gibbs measure first defined in \cite{LNS}, as well as an exit point bound under the annealed version of this measure. The bulk of the properties that we state in this section were proven already in \cite{LNS}. Finally we deduce our main results in Section \ref{sec:diff-tail} in an analogous argument to the discrete polymer case.

\paragraph{Acknowledgements.} The work of B.L. is supported by an NSERC Discovery grant. B.L. thanks Amol Aggarwal and Duncan Dauvergne for helpful and illuminating discussions. The work of P.S. is partially supported by NSF grants DMS-1811093 and DMS-2154090.

\paragraph{Competing Interests.} All authors certify that they have no affiliations with or involvement in any organization or entity with any financial interest or non-financial interest in the subject matter or materials discussed in this manuscript.

\section{Discrete models} \label{sec:disc-main}
We first consider four polymer models which were previously shown to be integrable: the log-Gamma polymer, the beta random walk, the strict-weak polymer and the inverse beta polymer of Thiery and Le Doussal. A common framework (the \emph{Mellin transform}) for the stationary version of these models was introduced by Chaumont and Noack \cite{CN}. We recall their setup here, and collect a few basic facts about the models which we use in our proofs.

\subsection{Partition function and stationarity}\label{sec: partition}
For $m,n\in \mathbb{Z}_+ :=\{ k \in \zz : k \geq 0 \}$ we let $\Pi_{m,n}$ be the set of up-right paths from $(0,0)\in \mathbb{Z}^2$ to $(m,n)$.  The environment of a general discrete polymer is a collection of weights $\{ \omega_e \}_e$ parameterized by the edges of $\mathbb{Z}_+^2$. The distribution of the weights for the polymers we consider will be specified below. Given the environment, the partition function is given by
\begin{equation}\label{eqn: Z}
Z_{m,n} := \sum_{x_\cdot \in \Pi_{m,n}} \prod_{i=1}^{m+n}
\omega_{e_i(x_\cdot)},
\end{equation}
where the $e_i (x_\cdot)$ denote the edges of the path $x_\cdot \in \Pi_{m, n}$, i.e.,  $e_i(x_\cdot):=x_i-x_{i-1}$.

We will make use of the following increments: for $k\in \{1,2\}$
\[R^k_x:=\frac{Z_x}{Z_{x-\alpha_k}}\qquad \text{for all  }x\text{ such that }  x-\alpha_k\in\mathbb{Z}_+^2.\]
where $\alpha_1=(1, 0)$, and $\alpha_2 = (0, 1)$ corresponding to horizontal or vertical increments, respectively. By definition $Z_{0,0} =1$.

For $x$ on the boundary of the quadrant $x \in \del \zz_+^2 := \{ (i, 0) : i \geq 0 \} \cup \{ (0, i ) : i \geq 0 \}$, the $R^k_x$ (where defined) coincide with the weight of the edge linking $x-\alpha_k$ to $x$. These boundary edge weights will later be taken to be iid, and so for example 
\[Z_{0,N}=\prod_{k=1}^N R_{k,0}^1 ,\]
is a product of iid random variables.

\paragraph{Distributional structure.}  The edge weights of the four integrable polymer models have a shared general distributional structure which we now detail. This structure depends on the specification of the distribution of three independent random variables $(R_1, R_2, X)$ which are different for each of the four models; the exact choices will be specified below. Given these distributions, we first state how the edge weights are generated. 

First, the weights along the  boundary of the quadrant differ from those in the bulk. The edge weights on the horizontal and vertical boundary edges are given by two independent families of iid random variables that are copies of the two random variables $R^1$ and $R^2$, respectively. 
 For each interior vertex $x \in \zz_+^2 \backslash \del \zz_+^2$ the weights of the incoming edges $(x-\alpha_1, x)$ and $(x-\alpha_2, x)$ 
are iid copies of $(Y_1, Y_2) \in \rr^2$ (we clarify that $(Y_1, Y_2)$ do  not necessarily have independent components - there is only independence between incoming edge weights of different interior vertices). The distribution of $(Y_1, Y_2)$ is  specified by the third random variable $X$, in a way that differs for each of the four models. The dependence of $(Y_1, Y_2)$ on $X$ appears in the last column Table \ref{table: integrable} below. 

In order to specify the distribution of $(R^1, R^2, X)$, we will use the \emph{Mellin transform} framework of \cite{CN}.

For any non-negative $f$ supported on $(0, \infty)$ we introduce the density form
\begin{equation}\label{eqn: mellin}
\rho_{f,a}(x):= M_f(a)^{-1}x^{a-1}f(x) ,
\end{equation}
where $M_f (a)$ is the normalization constant turning this into a probability measure on $(0, \infty)$.  The parameter $a$ will be taken to lie in $a \in D (M_f ) := \mathrm{interior} \{ a : 0 < M_f (a) < \infty\}$.

For $Z$ a random variable, we write $Z\sim \rho_{f,a}(x)\,\mathrm{d}x$, if $Z$ is distributed according to the density \eqref{eqn: mellin}. In this case, $\log Z$ has density
\[\e^{ax}f(\e^x)\,\frac{\mathrm{d}x}{M_f(a)}\]
on $\mathbb{R}$. 

In the case of the four integrable polymers, $f$ will be one of the functions
\begin{align*}
    \e^{-\beta/x}, \quad \e^{-\beta x}, \quad  (1-x)^{\beta-1}1_{\{0<x<1\}}, \quad \left(1-\frac{1}{x}\right)^{\mu-1}1_{\{x>1\}} ,  \quad  \left(\frac{x}{x+1}\right)^{(\beta+\mu)}.
\end{align*}
For these five functions $D(M_f)$ is $(0, \infty)$, $(-\infty, 0)$, $(0, \infty$), $(-\infty, 0)$ and $(-\beta-\mu, 0)$, respectively. 

Each of the four integrable polymer models is specified by the choice of two densities $f^1, f^2$ and three parameters $\theta, \mu, \beta$, with the latter two parameters appearing in the definition of the functions $f^i$.  Given these choices, the random variables $(R^1, R^2, X)$ are given by,
\beq \label{eqn:edge-weights}
(R^1, R^2,X)\sim m_{f^1}(a_1)\otimes m_{f^2}(a_2)\otimes m_{f^1}(a_3).
\eeq
The choices of $f^i$ and $a_i$ (the latter constants being specified in terms of $\theta$ and $\mu$), as well as $(Y^1, Y^2)$ in terms of $X$, leading to the four integrable stationary models are specified in Table \ref{table: integrable}.  Note that these choices specify the environment distribution completely.
\begin{table}[]
\centering
\begin{tabular}{l || c | c | c | c }
Model & $f^1$ &   $f^2$ & $(a_1, a_2, a_3) $ & $(Y^1,Y^2)$  \\
\hline
Inv-Gamma & $\e^{-\beta / x}$ & $\e^{-\beta/x}$ &  $(\theta - \mu, -\theta, -\mu)$ &$(X, X)$ \\
Gamma &$ \e^{-\beta x}$& $\left( 1 - \frac{1}{x} \right)^{\mu-1} \1_{ \{ x > 1 \} }$&$(\mu+\theta , - \theta, \mu )$ &  $(X, 1)$\\
Beta & $(1-x)^{\beta-1} \1_{ \{ 0 < x < 1 \} }$& $\left(1 - \frac{1}{x} \right)^{\mu-1} \1_{ \{ x > 1 \} }$  &$(\mu+\theta, - \theta, \mu)$ &  $(X, 1-X)$ \\
Inv-Beta & $ \left( 1 - \frac{1}{x} \right)^{\beta-1 } \1_{ \{ x > 1 \} }$ & $ \left( \frac{x}{x+1} \right)^{\beta+\mu}$ &$(\theta - \mu, -\theta, -\mu)$ & $(X, X-1)$ 
\end{tabular}
    \caption{Choices of $f^i$ and $a_i$ for the four integrable polymers in the Mellin transform framework, and interior edge weight distribution. The range of $\theta$ for the four models is $(0, \mu)$, $(0, \infty)$, $(0, \infty)$ and $(0, \mu)$, respectively.}
    \label{table: integrable}
\end{table}

Because we will need to vary certain parameters of our models, we recast the above discussion and formally specify our models as follows. First, fix $f^1, f^2$ and $\mu, \beta >0$ to be one of the four choices specified by Table \ref{table: integrable}. Then, let $a_3 = \pm \mu$ as appropriate. For any $(a, b) \in D(M_{f^1} ) \times D(M_{f^2} )$ we can define the edge weights according to \eqref{eqn:edge-weights} and $(a_1, a_2) = (a, b)$ and the interior edge weights in terms of $X$ according to Table \ref{table: integrable}. Then, the partition function with this environment will be denoted by $Z^{a,b}_{m,n}$.  The stationary case is when $a+b = a_3$ (this final choice being the degree of freedom afforded by the parameter $\theta$). In particular, we will view $f^1, f^2, \beta, \mu$ (and consequently $a_3$) as  fixed for the remainder of the paper. Throughout, we will need to vary $a$ and $b$ and so we leave them as free parameters in the model. We will always assume that $(a, b) \in D(M_{f^1} ) \times D(M_{f^2} )$; in general, constants degenerate as we approach the boundary of this domain but are otherwise uniform over compact subsets of the domain. 

We have the following decomposition of the log-partition function (here WNSE refers to the compass directions west, north, south, east),
\begin{equation}
\log Z^{a,b}_{m,n}=W_{n,m}+N_{n,m}=S_{n,m}+E_{n,m}\label{eqn: NSEW}
\end{equation}
where
\begin{equation} \label{eqn NSEW}
W_{n,m}:= \sum_{j=1}^n \log R^2_{0,j},\quad E_{n,m}:=\sum_{j=1}^n \log R^2_{m,j},\quad  N_{n,m}:=\sum_{i=1}^m \log R^1_{i,n}, \quad S_{n,m}:=\sum_{i=1}^m \log R^1_{i,0}.    
\end{equation}
The \emph{down-right property} (see Proposition 2.3 of \cite{CN}) implies that when $a+b = a_3$,
\begin{equation}
    \label{eqn: Burke}
    \text{ each of the sums } W_{n,m},\, N_{n,m},\, S_{n,m},\, E_{n,m} \text{ is a sum of i.i.d. random variables},
\end{equation}
and moreover $E_{n,m} \stackrel{d}{=} W_{n,m}$, and $N_{n, m} \stackrel{d}{=} S_{n,m}$ with the distribution of $W_{n,m}$ and $S_{n,m}$ being independent of $m$ and $n$ respectively.

\paragraph{The functions $M_f(a)$.}
From \eqref{eqn: mellin}, we have the form
\[M_f(a):= \int_0^\infty x^{a-1}f(x)\,\mathrm{d}x.\]
Setting
\[x=\e^{y},\]
we have
\[M_f(a)=\int_{\mathbb{R}}\e^{a y} f(\e^{y})\,\mathrm{d}y.\]
If $X\sim \rho_{f,a}(x)\d x$, we also have the formula
\begin{equation}\label{eqn: bob}\frac{\partial^k}{\partial a^k}M_f(a)=M_f(a)\mathbb{E}[(\log X)^k].
\end{equation}

The logarithmic derivatives
\[\psi_n^f(a)=\frac{\partial^{n+1}}{\partial a^{n+1}}\log M_f(a)\]
play an important role in our argument. For notational convenience, we also denote
\[\psi_{-1}^f(a):=\log M_f(a).\]
The following two identities  have important consequences
\begin{align*}
    \mathbb{E}[\log X]&=\psi_0^f(a),\\
    \mathrm{Var}(\log X)&=\psi_1^f(a).
\end{align*}
In particular, when $a+b = a_3$ we have $\ee[ \log Z^{a,b}_{m,n} ] = m \psi^{f^1}_0 (b) + n \psi^{f^2}_0 (a)$.

\paragraph{Coupling.} We will need to compare the models with varying parameters $(a, b)$. We now specify a coupling that will facilitate our analysis. Introduce,
\[
F_i (a, x) := \frac{1}{ M_f (a)} \int_0^x y^{a-1} f^i (y) \d y
\]
for $i=1, 2$ and let $H_i (a, \cdot)$ be the inverse of $F^i (a, \cdot)$ defined on $(0, 1)$.  For an infinite family of iid uniform $(0, 1)$ random variables $\{ U_i \}_{i \in \mathbb{Z}}$ we set for $j \geq1$,
\[
R^2_{0, j} := H_1 (b, U_{-j} ), \qquad R^1_{j, 0} := H_2 (a, U_{j} ) .
\]
When necessary we will denote the dependence of these random variables on the parameter $a, b$ by $R^2_{0, j} (b), R^1_{j, 0} (a)$.  The first logarithmic derivative of $H_i$ is given by $\del_a \log H_i (a, x) = L^i (a, H^i (a, x) )$ where 
\beq \label{eqn:weight-monotonicity}
L_i (a, x) = - \frac{1}{x \rho_{f,a}(x) } \mathrm{Cov} ( \log X , \1_{ \{ X \leq x \} } ) > 0
\eeq
where $X \sim \rho_{f, a} (x) \d x$.  In particular, the boundary weights depend on the parameters $a$ and $b$ in a monotonic fashion. 

\paragraph{Gibbs measure.} The Gibbs measure of the polymer in the environment specified by $m_{f^1}(a)\otimes m_{f^2}(b)\otimes m_{f^1}(a_3)$ will be denoted by,
$$
E^{a,b}_{m,n} [f] :=\frac{1}{Z_{m,n}^{ab}} \sum_{ x_\cdot \in \Pi_{mn}} f ( x_\cdot ) \prod_{i=1}^{n+m} \omega_{e_i (x_\cdot)} 
$$
for, e.g., bounded measureable $f$.  The variance and covariance with respect to the Gibbs measure will be denoted by $\Var^{a,b}_{m,n}$ and $\Cov^{a,b}_{m,n}$.  For events $\F$ we will use the notation $Q^{a,b}_{m,n} [ \F] = E^{a, b}_{m, n} [ \1_{\F} ] $ for the quenched measure.

\subsection{Rains-EJS formula} The stationarity and exponential structure \eqref{eqn: mellin} allow us to derive an exact expression for a generating function of the partition function with off-stationary initial data. This formula first appeared in last passage percolation in the work of Rains \cite{R}, but was recently re-introduced to the study of the model and used to great effect by Emrah-Janjigian-Sepp{\"a}l{\"a}inen \cite{EJS}. 


\begin{theorem}\label{thm: EJS} Let $(a, b) \in D(M_{f^1} ) \times D(M_{f^2} )$ satisfy $a+b = a_3$. 
Let $\lambda\in \mathbb{R}$ satisfy,
\[a-\lambda\in D(M_{f_1}), \quad b+\lambda\in D(M_{f_2}).\]
Then
\begin{align}\mathbb{E}[\e^{\lambda \log Z_{m,n}^{a-\lambda,b}}]= \exp\big(m(\psi^{f^1}_{-1}(a)-\psi^{f^1}_{-1}(a-\lambda))+ n(\psi^{f^2}_{-1}(b+\lambda)-\psi^{f^2}_{-1}(b))\big), \label{eqn: EJS-1}\\
\mathbb{E}[\e^{\lambda \log Z_{m,n}^{a,b-\lambda}}]= \exp\big(m(\psi^{f^1}_{-1}(a+\lambda)-\psi^{f^1}_{-1}(a))+ n(\psi^{f^2}_{-1}(b)-\psi^{f^2}_{-1}(b-\lambda))\big) .\label{eqn: EJS-2}
\end{align}
\end{theorem}

\begin{proof}
Write, using the second decomposition of \eqref{eqn: NSEW}:
\[\mathbb{E}[\e^{\lambda \log Z_{m,n}^{a-\lambda,b}}]\\
= \mathbb{E}[\e^{\lambda \sum_{i=1}^m \log R^1_{i,0}+\lambda E_{m,n}(a-\lambda,b)}].\]
We have denoted by $E_{m,n}(a-\lambda,b)$ the sum
\[ E_{m,n} (a-\lambda, b)=\sum_{j=1}^n \log R^2_{m,j}\]
in an environment with distribution
\[(R^1, R^2,X)\sim m_{f^1}(a-\lambda)\otimes m_{f^2}(b)\otimes m_{f^1}(a+b).\]
For $s \in D(M_{f^1})$, the joint density of $(\log R^1_{1,0}(s),\ldots,\log R^1_{m,0}(s) )$ is
\[
g(x_1,\dots,x_m)=\frac{\e^{s\sum_{k=1}^mx_k}}{M_{f^1}(s)^m}\prod_{k=1}^m f^1 (\e^{x_k}) .
\]
Therefore,
\[\mathbb{E}[\e^{\lambda \sum_{i=1}^m \log R^1_{i,0}(a-\lambda)+\lambda E_{m,n}^{a-\lambda,b}}]=\left(\frac{M_{f^1}(a)}{M_{f^1}(a-\lambda)}\right)^m \mathbb{E}[\e^{\lambda E_{m,n}(a,b)}].\]
By \eqref{eqn: Burke}, $E_{m,n}(a,b)$ is a sum of $n$ i.i.d. random variables whose distribution coincides with $W_{m,n}$ and so
\begin{align*}
    \mathbb{E}[\e^{\lambda E_{m,n} (a, b)}]&=\left(\int \e^{\lambda x}\e^{bx} \frac{f(\e^x)}{M_{f^2}(b)}\,\mathrm{d}x\right)^n\\
    &=\left(\frac{M_{f^2}(b+\lambda)}{M_{f^2}(b)}\right)^n\\
    &=\exp\big(n(\psi^{f^2}_{-1}(b+\lambda)-\psi^{f^2}_{-1}(b))\big).
\end{align*}
The formula \eqref{eqn: EJS-1} follows from this. The proof of \eqref{eqn: EJS-2} is almost identical, or may be deduced from \eqref{eqn: EJS-1}. \qed
 \end{proof}

\subsection{Taylor expansion}
Let 
\begin{equation}\label{eqn: e-def}
    \mfe =\mfe(a,b,m,n):=m-n\frac{\psi_1^{f_2}(b)}{\psi_1^{f_1}(a)}.
\end{equation}
Note that if we assume
\begin{equation}\label{eqn: char-direction}
|m- N \psi_1^{f_2}(b)|\le A N^{\frac{2}{3}}, \quad |n- N \psi_1^{f_1}(a)|\le A N^{\frac{2}{3}},
\end{equation}
for some $A \geq 0$ and asymptotic parameter $N$, 
then 
\[|\mfe(a,b,m,n)|\le CN^{\frac{2}{3}}.\]
As a corollary of Theorem \ref{thm: EJS}, we have the expansion
\begin{equation}\label{eqn: taylor}
\begin{split}
    &\mathbb{E}[\e^{\lambda(\log Z^{a-\lambda,b}_{m,n}-\mathbb{E}[\log Z_{m,n}^{a,b}])}]\\
    =&\exp\Big(-\frac{\lambda^2}{2}\psi^{f_1}_1(a)\cdot \mfe(a,b,m,n)+\frac{\lambda^3}{6}(m\psi_2^{f_1}(a) +n\psi_2^{f_2}(b))+(n+m)O(\lambda^4)\Big).
\end{split}
\end{equation}

\section{Derivatives} \label{sec:disc-derivs}
In this section, we state a few key monotonicity properties which we will be using.  For an upright path $x_\cdot \in \Pi_{m, n}$ we let $t_1$ be the $x$ coordinate of the rightmost vertex it touches on the $x$-axis and $t_2$ be the $y$ coordinate of the highest vertex it touches on the $y$-axis. We will need to differentiate the log partition function with respect to the parameters $a, b$; the quantities $t_i$ appear in the formulas for these derivatives.  We also define the following notation for weight of a path $x_\cdot = (x_i )_{i \leq i \leq m + n } \in \Pi_{m,n}$ by
\begin{align}
    W(a,b)(x_\cdot)&= \prod_{i=1}^{m+n} \omega_{x_{i-1},x_i}^{a,b} \notag \\
    &=\prod_{i=1}^{t_1} H_i(a)\prod_{j=1}^{t_2} H_j(b) \prod_{i = (t_1\vee t_2)+1}^{m+n} \omega_{x_{i-1},x_i}^{a,b}, \label{eqn:path-weight}
\end{align}
where the second formula will be useful in the calculations that follow.

\bep Let $a, b \in D (M_{f_1} ) \times D (M_{f_2} )$. Then,
\begin{align}
    \partial_a \log Z^{a,b}_{m,n}&= E^{a,b}_{m,n}\left[\sum_{i=1}^{t_1}L_{1}(R^1_{0,i})\right]&\ge 0 \label{eqn:discrete-deriv}
    \\
    \partial_a \partial_b \log Z^{a,b}_{m,n}&= \mathrm{Cov}^{a,b}_{m,n} \left(\sum_{i=1}^{t_1}L_{1}(R^1_{0,i}),\sum_{j=1}^{t_2}L_{2}(R^2_{j,0})\right) &\le 0 \label{eqn:cov-sign}
\end{align}
Here, $L_{i} (x) = L_i (a, x)$. 
\eep
\begin{proof} The equalities are proven by the following direct calculation. 
Differentiating the formula \eqref{eqn:path-weight} for the weight of a path $x_\cdot \in \Pi_{m,n}$ we find,
\[\partial_a W(a,b)(x_\cdot)= \sum_{i=1}^{t_1} \frac{\partial_a H_i(a)}{H_i(a)} \cdot W(a,b)(x_\cdot) \] 
The first identity follows from summing over paths. The second follows by a similar calculation differentiating the above formula for $\partial_a W(a, b)$ wrt to $b$ and again summing over the paths. 

For the claimed inequalities, the first follows from the fact that the $L_i$ are all positive. The second inequality follows because only one of $t_1$ or $t_2$ can be non-zero and again that the $L_i$ are all positive. \qed
\end{proof} 


\begin{prop} \label{prop:general-increasing}
Let $g: \mathbb{R}_+\rightarrow \mathbb{R}_+$ be an increasing function. Then, the expectation $E_{m,n}^{a,b}[g(t_1)]$ with respect to the polymer measure is an increasing function of $a$. 
\end{prop}
\begin{proof}
Recalling the notation \eqref{eqn:path-weight} we write,
\[E_{m,n}^{a,b}[g(t_1)]=\frac{1}{Z^{a,b}_{m,n}}\sum_{x_\cdot \in \Pi_{m,n}} g(t_1(x_\cdot)) W(a,b)(x_\cdot),\]
and obtain by differentiation,
\begin{equation}\label{eqn: monotonicity-E}
\partial_a E_{m,n}^{a,b} [g(t_1)]=\mathrm{Cov}^{a,b}_{m,n} \left( g(t_1),\sum_{i=1}^{t_1} L_i(a) \right)=\sum_{i=1}^m L_i(a) \cdot \mathrm{Cov}^{a,b}_{m,n} \left( g(t_1), 1_{\{t_1\ge i\}} \right)\ge 0.
\end{equation}
The last inequality follows from positive association and $L_i(a)\ge 0$. \qed
\end{proof}

\section{Tail bound for $t_1$} 
\label{sec:disc-exit}
\begin{theorem} \label{thm:t1-bd-1} 
Let $0<\varepsilon_0\le 1$ be small.  Assume the parameters $a,b,m,n$ and $0 \leq \lambda_1 \leq \varepsilon_0$ are such that for some $C_1 >0$ we have,
\begin{equation}\label{eqn: characteristic}
    | m\psi_1^{f_1}(a+2 \lambda_1)-n\psi_1^{f_2}(b- 2\lambda_1)| \leq C_1 ,
\end{equation}
and $a+b = a_3$. 
Then, we have
\[\mathbb{E}[ Q^{a,b}_{m,n}[ t_1 > 0 ]]\le \exp\left( -\frac{1}{2}\lambda_1^3 (m\psi_2^{f_1}(a+2 \lambda_1)+n\psi_2^{f_2}(b-2 \lambda_1))+C(m+n)\lambda_1^4 +C_1 \lambda_1^2\right).\]
\end{theorem}
\begin{proof}
We first claim the following estimate for the upper tail of the exit point $t_1$: for any $0 \le r\le 1$, we have, for $\lambda_1\ge 0$, $\lambda_2 \in \rr$,
\begin{equation}
    Q^{a,b}_{m,n}[  t_1 > 0  ] \le  Q^{a,b}_{m,n}[ t_1 > 0  ]^r\le  \left(\frac{Z_{m,n}^{a+\lambda_1,b+\lambda_2}}{Z_{m,n}^{a+\lambda_1,b}}\right)^r.
\end{equation}
The first inequality follows since $x \le x^r$ for $0\le x \le 1$. For the second, note that by Proposition \ref{prop:general-increasing} with $g(x) = \1_{ \{ x > 0 \} }$ we first have
\[ Q^{a,b}_{m,n}[ t_1 > 0  ] \le  Q^{a+\lambda_1,b}_{m,n}[ t_1 > 0].\]
Next if $t_1>0$, then modifying the weights on the second axis has no effect:
\begin{align*}
   Q^{a+\lambda_1,b}_{m,n}[ t_1 > 0]&=\frac{1}{Z^{a+\lambda_1,b}_{m,n}}\sum_{x_\cdot \in \Pi_{m,n}} \1_{ \{t_1 ( x_\cdot ) > 0 \} } W(a+\lambda_1,b)(x)\\
    &=\frac{1}{Z^{a+\lambda_1,b}_{m,n}}\sum_{x_\cdot \in \Pi_{m,n}} \1_{ \{t_1 ( x_\cdot ) > 0 \} }  W(a+\lambda_1,b+\lambda_2)(x)\\
    &\le \frac{Z_{m+n}^{a+\lambda_1,b+\lambda_2}}{Z_{m,n}^{a+\lambda_1,b}}.
\end{align*}
We thus have
\[\mathbb{E}[  Q^{a+\lambda_1,b}_{m,n}[ t_1 > 0  ] ]\le \mathbb{E}\left[ \e^{r(\log Z_{m,n}^{a+\lambda_1,b+\lambda_2}-\log Z_{m,n}^{a+\lambda_1,b})} \right].\]
Choose $2r=\lambda_1$ and $\lambda_2=-2\lambda_1$. Applying Cauchy-Schwarz and Theorem \ref{thm: EJS} we have,
\begin{align*}
    \log \mathbb{E}[\e^{r(\log Z_{m,n}^{a+\lambda_1,b+\lambda_2}-\log Z_{m,n}^{a+\lambda_1,b})}]^2 &\leq \log \ee [ \e^{ \lambda_1 \log Z_{m,n}^{a+\lambda_1, b - 2 \lambda_1 } } ]  \ee[ \e^{ - \lambda_1 \log Z^{a+\lambda_1, b}_{m,n}}] \\
    &= m \left( \psi_{-1}^{f^1} (a+2 \lambda_1 ) + \psi_{-1}^{f^1} (a) - 2 \psi_{-1}^{f^1} (a + \lambda_1 ) \right) \\
    &+ n \left( 2 \psi_{-1}^{f^2} ( b - \lambda_1 ) - \psi_{-1}^{f^2} ( b - 2 \lambda_1 ) - \psi_{-1}^{f^2} (b) \right) \\
    &\leq -\lambda_1^3 ( m \psi_2^{f^1} (a +2 \lambda_1) + n \psi_2^{f^2} (b-2 \lambda_1) ) + C (m+n) \lambda_1^4 +C_1 \lambda_1^2.
\end{align*}
where the final inequality follows by a Taylor expansion and \eqref{eqn: characteristic}. The claim follows. \qed

From Lemma C.2 of \cite{CN} and smoothness of the $\psi_i^{f^i}$ we have the following.
\bel \label{lem:cn-sign}
Let $a, b$ satisfy $a+b = a_3$. Then,
\[
\psi_1^{f^1} (a) \psi_2^{f^2} (b) + \psi_1^{f^2} (b) \psi_2^{f^1} (a) > 0
\]
and this quantity is uniformly bounded away from $0$ for $a + b =a_3$ varying over compact sets.
\eel

\end{proof}

\begin{theorem}
Assume that $c_1 N\le \max\{m,n\}\le C_1 N$ for some $c_1, C_1 >0$. 
Recall the definition of $\mfe(a,b,m,n)$ in \eqref{eqn: e-def}. Assume $a+b = a_3$. 
There is a $0<\delta<1$  and $C, c>0$ such that
\begin{equation}\label{eqn: upper-bd}
\mathbb{E}[ Q [ t_1> \mfe+w  ]] \le  C \exp\Big(- \frac{w^3}{c^2 N^2}+CN^{-3}w^4\Big),
\end{equation}
valid for $0 \le w \le \delta N$ and $| \mfe | \le \delta N$ and $0 \leq \mfe + w$.
\end{theorem}
\begin{proof}
We have the stationarity statement \cite[Lemma 5.1]{CN}
\[  Q_{m,n}^{a,b}[ t_1> k ] \stackrel{d}{=}  Q_{m-k,n}^{a,b}[ t_1> 0 ].\]
We apply this with,
\[k=\lfloor \mfe+w \rfloor,\]
 and apply Theorem \ref{thm:t1-bd-1} to the latter quenched probability. We wish to choose $\lambda_1 >0$ satisfying,
 \beq \label{eqn:lambda-def}
 \frac{\psi^{f^2}_1(b)}{\psi^{f^1}_1(a)}-\frac{\psi_1^{f^2}(b-2\lambda_1)}{\psi_1^{f^1}(a+2\lambda_1)}=\frac{w}{n}.
 \eeq
The derivative of the LHS wrt $\lambda_1$ at $\lambda_1 = 0$ equals \[c(a,b):=2 \frac{\psi_2^{f^2}(b)\psi_1^{f^1}(a)+\psi_2^{f^1}(b)\psi_1^{f^2}(a)}{(\psi_1^{f^1}(a))^2}>0,\]
 by Lemma \ref{lem:cn-sign}. Therefore, for $w\leq \delta N$ for $\delta >0$ sufficiently small, the equation \eqref{eqn:lambda-def} has a solution satisfying $c(a, b) \lambda_1 = wn^{-1} + \O ( w^2 n^{-2} )$. A straightforward calculation using the definition of $\mfe$ shows that 
 \beq \label{eqn:char-dir-2}
 | (m-k) \psi_1^{f^1} (a + 2 \lambda_1 ) - n \psi_1^{f^2} ( b - 2 \lambda_1 ) | \leq C_2
 \eeq
 for some $C_2 >0$ independent of $w$.

 Theorem \ref{thm:t1-bd-1} now implies
 \begin{align*}
     \ee\left[ Q_{m-k,n}^{a,b}[ t_1> 0  ] \right] &\leq C \exp \left( - \frac{w^3}{ 2 c(a, b)^3 n^3}  \left( (m-k) \psi_2^{f^1} (a + 2 \lambda_1 ) + n \psi_2^{f_2} ( b - 2 \lambda_1 ) \right) + C w^4 N^{-3} \right) \\
     &\leq C  \exp \left( - \frac{w^3 \psi_1^{f^1} (a) }{ 4 c(a, b)^2 n^2}   + C w^4 N^{-3} \right) 
 \end{align*}
 where we used \eqref{eqn:char-dir-2} and smoothness of the functions in $\lambda_1$. The claim now follows. \qed

\end{proof}

From the previous result, we obtain 
\bec \label{cor:disc-tail-1} Assume that $c_1 N\le \max\{m,n\}\le C_1 N$ for some $c_1, C_1 >0$ and $a+b = a_3$. 
For any $\lambda \geq 0$, we have the estimate
\begin{equation}\label{eqn: three-two}
    \mathbb{E}[E_{m,n}^{a,b}[\e^{\lambda t_1}]]\le C\e^{C\lambda^{\frac{3}{2}} N+\lambda \mfe_+}.
\end{equation}
Here $\mfe_+=\max\{0,\mfe\}$ is the positive part of $\mfe$.
\eec
\begin{proof}
Write
\begin{align*}\mathbb{E}[\e^{\lambda t_1}]-1&\le\lambda \int_0^{\mfe_+} \e^{\lambda s}\,\mathrm{d}s+\lambda \int_{\mfe_+}^{\delta N}\e^{\lambda s} \mathbb{E}[ Q^{(a, b)}_{m,n} [ (t_1>s ]] \,\mathrm{d}s
+\lambda \int_{\delta N}^{CN}  \e^{\lambda s}\mathbb{E} [ Q^{a, b}_{m,n} [ t_1>\delta N]]\,\mathrm{d}s\\
&\le \e^{\lambda \mfe_+}+\lambda \e^{\lambda \mfe_+}\int_{0}^{\delta N-\mfe_+} \e^{\lambda w}\exp\Big(-\frac{w^3}{c^2N^2}+CN^{-3}w^4\Big)\,\mathrm{d}w+\e^{C \lambda N-\delta c N}.
\end{align*}
In the second step we have estimated $\mathbb{E}[ Q^{(a, b)}_{m,n} [ t_1>\delta N]]$ using \eqref{eqn: upper-bd} with $\delta$ sufficiently small. 
The last term $\e^{C \lambda N-\delta c N}$ is neglible for $C \lambda < \delta c$ and bounded by $\e^{\lambda^{3/2}C^{1/2} (\delta c)^{-1/2}N}$ for $C \lambda\ge \delta c$. For $\delta$ sufficiently small, the middle integral is bounded by
\[\int_0^\infty \e^{\lambda w-\frac{w^3}{2cN^2}}\,\mathrm{d}w\le C\e^{N \lambda^{\frac{3}{2}}}.\]
The claimed result is now clear. \qed
\end{proof}

\section{Tail bound} \label{sec:disc-tail}
We can now proceed to the tail bound. We will assume that \eqref{eqn: char-direction} holds. 
\begin{theorem} \label{thm:discrete-main-tech}
Assume \eqref{eqn: char-direction}, and $a+b = a_3$. There is a constant $\eps_0 >0$ so that
\begin{equation}\label{eqn: upper-lower}
c\e^{c\lambda^3N}\le \mathbb{E}[\exp\big(\lambda(\log Z_{m,n}^{a,b}-\mathbb{E}[\log Z_{m,n}^{a,b}])\big)]\le C\e^{C\lambda^3N}
\end{equation}
holds for all $0 \leq \lambda \leq \eps_0$. 
\end{theorem}

\begin{proof}We first write:
\begin{equation}
\label{eqn: FTC}
\log Z_{m,n}^{a,b}=\log Z_{m,n}^{a-\lambda,b}+\int_{a-\lambda}^a \partial_s \log Z_{m,n}^{s,b}\,\mathrm{d}s.
\end{equation}
By Cauchy-Schwarz, we have
\[\mathbb{E}[\exp((\lambda/2) \log Z_{m,n}^{a,b})]^2\le \mathbb{E}[\exp(\lambda \log Z_{m,n}^{a-\lambda,b})]\mathbb{E}[\exp\big(\lambda \int_{a-\lambda}^a \partial_s \log Z_{m,n}^{s,b}\,\mathrm{d}s\big)].\]
Using the Taylor expansion \eqref{eqn: taylor}, the first factor is bounded as follows
\[\mathbb{E}[\exp(\lambda \log Z_{m,n}^{a-\lambda,b})]\le \e^{\lambda \mathbb{E}[\log Z_{m,n}^{a,b}]}\cdot\e^{ C\lambda^3N + \mfe \lambda^2} \leq \e^{\lambda \mathbb{E}[\log Z_{m,n}^{a,b}]}\cdot\e^{ C\lambda^3N }.\]
for $|\lambda|<1$ sufficiently small, using
\beq \label{eqn:assump-imp}
\lambda^2 \mathfrak{e}_+\le \frac{1}{3}(2C_1^{3/2}\lambda^3 N+1),
\eeq
which holds under \eqref{eqn: char-direction}. 
Next, consider the integral
\beq \label{eqn:disc-tail-a1}
\mathbb{E}\left[\exp\left(\lambda \int_{a-\lambda}^a \partial_s \log Z_{m,n}^{s,b}\,\mathrm{d}s\right)\right]=\mathbb{E}\left[\exp\left(\lambda \int_{a-\lambda}^a E_{n,m}^{s,b}\Big[\sum_i^{t_1} L_i\Big]\,\mathrm{d}s\right)\right],
\eeq
where we used \eqref{eqn:discrete-deriv} and for the remainder of the proof we let $L_i := L_1 (R_{0,i}^1 ).$ We recall now that for each value of $(s, b)$ the $\{ L_i \}_{i=1}^m$ are iid. For a random variable $X$ we use the notation $\overline{X} := X- \ee[X]$. We center,
\[L_i=\mathbb{E}[L_1]+\overline{L}_i.\]
The RHS of \eqref{eqn:disc-tail-a1} is bounded above by,
\[\mathbb{E}\Big[\e^{C\lambda^2 E^{a,b}_{n,m}[t_1]}\exp\big(\lambda \int_{a-\lambda}^a E^{s,b}_{n,m}[\overline{S}]\,\mathrm{d}s\big)\Big]\le \mathbb{E}[ E^{a,b}_{n,m} [ \e^{2C\lambda^2 t_1}] ]^{\frac{1}{2}}\mathbb{E}\Big[\exp\big(2\lambda \int_{a-\lambda}^a E^{s,b}_{n,m}[\overline{S}]\,\mathrm{d}s\big)\Big]^{\frac{1}{2}},\]
where $\overline{S} = \sum_{i=1}^{t_i} \overline{L}_i$ denotes the centered sum and we have used Proposition \ref{prop:general-increasing}. By Jensen's inequality, we have
\[\mathbb{E}\Big[\exp(2\lambda \int_{a-\lambda}^a E_{n,m}^{s,b}[\overline{S}])\,\mathrm{d}s\Big]\le \frac{1}{\lambda}\int_{a-\lambda}^a \mathbb{E}[E_{n,m}^{s,b}[\e^{2\lambda^2 \overline{S}}]]\,\mathrm{d}s.\]
Let now $\overline{S}_i = \sum_{k=1}^i \overline{L}_k$. Then, following the method in \cite[Lemma 4.2]{S}, we have for any $C' >0$,
\begin{align*}
E_{n,m}^{s,b}[e^{2\lambda^2 \overline{S}}]= &\sum_{i=1}^m \e^{2\lambda^2 \overline{S}_i}Q^{s,b}_{m,n}(t_1=i)\\
    \le&~ \sum_{i=1}^m \e^{C' \lambda^2 i} Q^{s,b}_{m,n}(t_1=i)+\sum_{i=1}^m \e^{\lambda S_i} \1_{ \{S_i> C' i\} }\\
    \le&~ E^{s,b}_{m,n}[\e^{C' \lambda^2 t_1} ]+X(C',s)\\
    \le&~ E^{a,b}_{m,n}[\e^{C'\lambda^2 t_1}]+X(C',s),
\end{align*}
where $\mathbb{E}[X(C',s)]$ is bounded uniformly in $s$ for large enough $C'$, since $L_i$ has a finite exponential moment for each of the four integrable models (See \cite[Remark 3.7]{CN}). Note that in the last inequality we again used Proposition \ref{prop:general-increasing}. In the end, we obtain an estimate of the form
\[\mathbb{E}\Big[\exp\left(2\lambda \int_{a-\lambda}^a \partial_s \log Z_{m,n}^{s,b}\,\mathrm{d}s\right)\Big]\le C\mathbb{E}\big[E^{a,b}_{m,n}[\e^{C\lambda^2 t_1}]+1\big]\le C\e^{C\lambda^3 N+\lambda^2 \mathfrak{e}_+},\]
for some constant $C$. 
In the last step we have used \eqref{eqn: three-two}. Using again \eqref{eqn:assump-imp} that holds under the assumption \eqref{eqn: char-direction}, we have 
obtained the upper bound in \eqref{eqn: upper-lower}. The lower bound follows by dropping the second term on the right side of \eqref{eqn: FTC} and applying \eqref{eqn: taylor} to find
\[\mathbb{E}[\exp(\lambda \log Z_{m,n}^{a-\lambda,b})]\ge \e^{\lambda \mathbb{E}[\log Z_{m,n}^{a,b}]}\cdot\e^{c\lambda^3N  - \lambda^2 \mfe}\]
for $\lambda$ small enough.  We then use that for any $\eps >0$ we have $\lambda^2 \mfe \leq \eps \lambda^3 N + C_\eps$ under assumption \eqref{eqn: char-direction}. \qed
\end{proof}

\noindent{\bf Remark.} Under the above assumptions one can obtain the estimate,
\[
\mathbb{E}[\exp\big(-\lambda(\log Z_{m,n}^{a,b}-\mathbb{E}[\log Z_{m,n}^{a,b}])\big)] \leq C \e^{ C \lambda^3 N}
\]
via a similar proof. Indeed, instead of \eqref{eqn: FTC} one may write,
\[
\log Z^{a+\lambda, b}_{n, m} \leq \log Z^{a, b}_{n, m} + \int_0^\lambda \del_{s_1} \log Z^{s_1,s_2}_{n,m} \vert_{s_1= a+s, s_2 = b-s} \d s
\]
using \eqref{eqn:cov-sign}. Then, following the proof line-by-line one finds the need to estimate, for $0 < s < \lambda$,
\[
\ee \left[ E^{a+s,b-s}_{n,m} [ \e^{ C \lambda^2 t_1 } ] \right] \leq C \e^{ C \lambda^3 N + C \lambda^2 \mfe ( a+s, b-s, n, m )_+ } \leq C \e^{ C \lambda^3 N}
\]
with the last inequality following from $|\mfe ( a+s, b-s, n, m ) - \mfe (a, b, n, m) | \leq C N \lambda$ and \eqref{eqn: char-direction}. Everything else is identical.   \qed

\section{Diffusion model and statement of results} 
\label{sec: diffusion}  
We consider the following system for $N$ diffusions $ \{ u_i (t) \}_{i=1}^N \in \rr^N$,
\begin{align}
\d u_1 &= - V' (u_1 ) \d t  + \d B_0 - \theta \d t + \d B_1 \notag \\
\d u_j &= - V' (u_j ) \d t  + V' (u_{j-1} ) \d t + \d B_j - \d B_{j-1} , \qquad j \geq 2 .
\end{align}
Above, the $\{ B_j (t) \}_j$ are a family of independent standard Brownian motions. 
The class of potentials we consider is the following, the name ``O'Connell-Yor'' will be made clear momentarily.
\bed \label{def:V-def} 
We say $V$ is of O'Connell-Yor-type if $V \geq 0$ is a smooth convex function satisfying,
\beq \label{eqn:assumption}
V(x) \geq c |x|^2 \1_{ \{ x \leq - C \} } , \qquad V'(x) \leq 0
\eeq
and 
\beq \label{eqn:deriv-assump}
c_0 V''(x) \leq - V''' (x) \leq \frac{1}{c_0} V'' (x) + C \1_{ \{ x \geq -C \} }
\eeq
\eed
Under these assumptions, the above system is a Markov process admitting global-in-time strong solutions with a unique invariant measure $\omega_\theta$ of product form, \cite[Proposition 2.2]{LNS},
\beq \label{eqn: stationary}
\d \omega_\theta (u) := \prod_{i=1}^N \d \nu_\theta ( u_i ) := \prod_{i=1}^N   \frac{\e^{ - \theta u_i - V (u_i)} }{ Z ( \theta ) } \d u.
\eeq

The observable we consider is the following.
\bed Let $\{ u_j (t) \}_{j=1}^N$ be the solution to \eqref{eqn:sys} with initial data distributed according to the invariant measure $\omega_\theta$.  The height function is defined by,
\[
W^\theta_{N, t} := \sum_{j=1}^N u_j (t) - B_0 (t) + \theta t.
\]
\eed
In the special case $V(u) = \e^{-u}$, it is well-known that this height function corresponds to the log partition function of the stationary O'Connell-Yor polymer \cite{OY}. For the reader's convenience, we recall that this is given by,
\[
 Z^{OY}_{N, t} (\theta) := \int_{-\infty < s_0 < \dots < s_{N-1} < t } \e^{ \theta s_0 - B_0 (s_0) + \sum_{j=1}^N B_j (s_j ) - B_j (s_{j-1} ) } \d s
\]
where the $\{ B_j \}_j$ are extended to two-sided Brownian motions equal to $0$ at $t=0$ and we use the convention $s_N = t$. We will make no use of this representation in proving our main results.

A large class of potentials $V(u)$ satisying the assumptions of Definition \ref{def:V-def} are given by Laplace transforms of finite, positive measures with support compactly supported in $(0, \infty)$, as well as their small perturbations by compactly supported smooth functions.

Given $V$ as above, we define for $k \geq -1$,
\[
\psiV_k (\theta ) := \frac{ \d^{k+1} }{ \d \theta^{k+1} } \log Z  ( \theta ) .
\]
We remark that,
\beq \label{eqn:W-expect}
\ee[ W_{N, t}^\theta ] = \theta t - N \psiV_0 ( \theta ).
\eeq
Later we will see that a special role is played by the function,
\beq \label{eqn:e-def}
\mff ( \theta , t ) := t -N \psiV_1 ( \theta ) .
\eeq
We will prove the following tail estimate for the height function.
\bet \label{thm:main-diff}
 Let $V$ be a potential of O'Connell-Yor type and $\theta$ satisfy 
 \beq \label{eqn:curvature}
 \psiV_2 ( \theta) < 0 .
 \eeq
 Suppose there is an $A >0$ so that,
\beq \label{eqn:main-diff-1}
|N \psiV_1 ( \theta ) - t | \leq A N^{2/3}.
\eeq
Then, there is a $c>0$ so that for all $ c N^{2/3} \geq s>0$,
\[
\pp\left[ | W^\theta_{N, t} - \ee[ W^\theta_{N, t} ] | > s N^{1/3} \right] \leq c^{-1} \e^{ - c s^{3/2} }
\]
and a $C >0$ so that
\[
\pp\left[  W^\theta_{N, t} -\ee[ W^\theta_{N, t} ]  > s N^{1/3} \right] \geq C^{-1} \e^{ - C s^{3/2} }
\]
\eet
\remark If one wishes to drop the requirement \eqref{eqn:main-diff-1}, our proof yields instead the estimates, 
\beq
\pp\left[ | W^\theta_{N, t} - \ee[ W^\theta_{N, t} ] | > s N^{1/3}  \right] \leq C \begin{cases} \e^{ - c N^{2/3} s^2 | \mff ( t , \theta ) |^{-1} }, &  s \leq |\mff (t , \theta ) |^2 N^{-4/3} \\
\e^{ - c s^{3/2} } , &  s >|\mff (t , \theta ) |^2 N^{-4/3} \end{cases}
\eeq
and
\beq
\pp\left[  W^\theta_{N, t} - \ee[ W^\theta_{N, t} ]  > s N^{1/3}  \right] \geq c \begin{cases} \e^{ - C N^{2/3} s^2 | \mff ( t , \theta ) |^{-1} }, &  s \leq |\mff (t , \theta ) |^2 N^{-4/3}  \\
\e^{ - C s^{3/2} } , & s >|\mff (t , \theta ) |^2 N^{-4/3}  \end{cases}
\eeq
for $ 0 \leq s \leq c N^{2/3}$. In fact, one can obtain the same Gaussian tail as a lower bound for the event $\{ W^\theta_{N, t} - \ee[ W^\theta_{N, t} ] < -s N^{1/3} \}$ for $s < c |\mff (t , \theta ) |^2 N^{-4/3} $, some $c>0$. \qed 
\section{Diffusions preliminaries} \label{sec:diff-prelim}

\subsection{Couplings}

We will need to consider the system for many parameter values simultaneously. In this section we introduce the couplings we will use. Let $F_\theta$ denote the cumulative distribution function,
\[
F_\theta (x) := \int_{-\infty}^x \d \nu_\theta (u)
\]
and let $H_\theta (u)$ denote its inverse. By the implicit function theorem, this is a smooth function in $u$ and $\theta$.  For notational simplicity, we choose a realization of the Brownian motions so that they are continuous for every point in the underlying probability space.  In \cite[Proposition 3.1]{LNS} we showed that for every choice of initial data there exists a solution to \eqref{eqn:sys} for every realization of the Brownian motions, in that there are continuous functions $u_j (t)$ satisfying the integrated form of these equations. Moreover, this process is a Markov process with unique invariant measure $\omega_\theta$.

Let now $\{ U_j \}_{j=1}^N$ be a sequence of iid uniform $(0, 1)$ random variables.  For $\eta >0$ and $\theta >0$ let $u_j (t, \eta, \theta)$ for $ t  \geq 0$ be defined by,
\begin{align*}
\d u_1 &= - V' (u_1 ) \d t  + \d B_0 - \theta \d t + \d B_1  \\
\d u_j &= - V'' (u_j ) \d t  - V' (u_{j-1} ) \d t + \d B_j - \d B_{j-1} , \qquad j \geq 2 .
\end{align*}
with initial data $u_j (0, \eta, \theta )  = H_\eta ( U_j )$.  We will denote,
\beq \label{eqn:W-def}
W_{N, t} (\eta, \theta ) := \sum_{j=1}^N u_j (t, \eta, \theta ) - B_0 (t) + \theta t 
\eeq
so that for $t>0$ we have that $W_{N, t} (\theta, \theta)$ has the same distribution as $W_{N, t}^\theta$.

\paragraph{Remark.} In the arguments that follow, it may be helpful to the reader to consider the O'Connell-Yor case in which an essentially equivalent two-parameter model may be realized by
\[
Z^{OY}_{N, t} ( \eta , \theta ):=  \int_{-\infty < s_0 < \dots < s_{N-1} < t } \e^{ -\eta (s_0)_- +\theta (s_0)_+ - B_0 (s_0) + \sum_{j=1}^N B_j (s_j ) - B_j (s_{j-1} ) } \d s .
\]
For example, derivatives of $\log Z^{OY}_{N, t}$ wrt $\eta, \theta$ are easily computed in terms of quenched moments of $(s_0)_-$ and $(s_0)_+$ and their signs are deduced in a trivial fashion.  A tail estimate for $(s_0)_+$ can then be deduced via almost identical arguments to those in Section \ref{sec:disc-exit}, and the tail estimate for $\log Z^{OY}_{N, t}$ is the same as in Section \ref{sec:diff-tail}. The difference between this case and the more general diffusions model is that the montonicity properties are more transparent in the OY case. In fact, the arguments for the OY polymer may be more transparent than even the discrete polymer case, as there is no appearance of the logarithmic derivatives $L_i(a)$ which introduced complications in deducing the tail estimate for the discrete polymer free energy from the exit point bounds. \qed 

\subsection{Derivatives}

We will need the following.
\bep \label{prop:diff-derivs}
The functions $u_j (t, \eta, \theta)$ and $W_{j, t} (\eta, \theta)$ are $C^2$ in the parameters $(\eta, \theta)$ and satisfy the inequalities,
\beq
\del_\eta W_{N, t} (\eta, \theta) \leq 0 , \qquad \del_\theta W_{N, t}  (\eta, \theta) \geq 0
\eeq
and
\beq
\del_\eta \del_\theta W_{N, t} (\eta, \theta) \geq 0 , \qquad \del_\theta^2 W_{N, t} (\eta, \theta) \geq 0.
\eeq
\eep
\proof Differentiability follows from \cite[Proposition 5.1]{LNS}. The inequalities for the first derivatives follow from \cite[Proposition 5.3]{LNS}. The signs for the second derivatives follow from \cite[Lemma 4.3]{LNS}. \qed

\subsection{Generating function}

We require the following \cite[Proposition 6.1]{LNS}.  It is the version of the Rains-EJS identity for our systems of diffusions. 
\bep \label{prop:generat}
Let $W_{N, t} ( \eta, \theta)$ be as in \eqref{eqn:W-def} and define,
\[
\varphi ( \theta ) := N \psiV_{-1} ( \theta ) - \frac{1}{2} \theta^2 t = N \log (Z( \theta)) - \frac{1}{2} \theta^2 t.
\]
Then,
\[
\ee\left[ \exp \left(  ( \eta - \theta ) W_{N, t} ( \eta , \theta ) \right) \right] = \exp \left( \varphi ( \theta ) - \varphi ( \eta ) \right).
\]
\eep

Using this, we derive the following corollary.
\bec \label{cor:diff-generat}
For $\eta, \theta >0$ we have
\begin{align*}
 & \ee\left[ \exp \left ( \eta - \theta ) W_{N, t} ( \eta , \theta ) \right) \right] \notag \\
= & \exp \left( ( \eta - \theta ) \ee[ W_{N, t} ( \eta , \eta ) ] - \frac{ ( \theta - \eta)^2}{2} \mff ( \eta , t) + \frac{ ( \theta - \eta)^3}{6} N\psiV_2 ( \eta ) + \O ( N ( \eta - \theta)^4 ) \right)
\end{align*}
\eec
\proof We have by Taylor expansion,
\begin{align*}
&N ( \psi_{-1}^V ( \theta ) - \psi_{-1}^V ( \eta ) ) + \frac{1}{2} t ( \eta^2 - \theta^2) \notag \\
= &( \theta - \eta ) ( N \psiV_0  ( \eta ) - t \eta ) + \frac{ ( \theta - \eta)^2}{2} \left( N \psiV_1 ( \eta ) - t \right) + \frac{ ( \theta - \eta)^3}{6} \psiV_2 ( \eta) \notag\\
+ & \O ( N (\theta - \eta^4 ))
\end{align*}
The claim follows from \eqref{eqn:W-expect} and \eqref{eqn:e-def}. \qed

Similar to the proof of Theorem 2.2 of \cite{EJS}, one can deduce the following bound for the wedge-initial condition O'Connell-Yor polymer, which can be written in terms of the two parameter model as $Z^{OY}_{N, t} ( \infty, 0)$.

\bec \label{cor:oy-non-stat}
Let $c>0$ so that $c N \leq t \leq c^{-1} N$. Let $\theta$ satisfy $\psiV_1 (\theta_0 ) = tN^{-1}$ for $V = \e^{-x}$. Then,
\[
\pp \left[ \log Z_{N, t}^{OY} ( \infty, 0) > u + \theta_0 t- N\psiV_0 ( \theta_0 ) \right] \leq \exp\left( - N \frac{ 2 \sqrt{2}}{3} \frac{ u^{3/2}}{| \psiV_2 (\theta_0 ) |^{1/2} } + C N u^2\right)
\]
for all $0 \leq u \leq N$.
\eec
\proof For any $(\eta, \theta)$ we have $Z^{OY}_{N, t} ( \infty, 0) \leq Z^{OY}_{N, t} ( \eta , \theta)$. Therefore by Markov's inequality and any $\theta, a >0$ we have by Proposition \ref{prop:generat},
\[
\pp\left[ \log Z_{N, t}^{OY} ( \infty, 0) > s \right] \leq \e^{ - a s} \e^{ \varphi ( \theta) - \varphi ( \theta + a ) }.
\]
The claim follows from the choice $\theta = \theta_0$ (which makes the quadratic terms in the Taylor expansion of $\varphi ( \theta) - \varphi ( \theta + a )$ vanish) and $a = (2u)^{1/2} | \psiV_2 ( \theta_0) |^{-1/2}$ where $s = u +  (\theta t- N\psiV_0 ( \theta_0 ) )$ (which optimizes in $a$ between the linear term and cubic terms in the above estimate), and a Taylor expansion to third order of $\varphi ( \theta + a)$ around $\theta = \theta_0$. \qed

\subsection{Pseudo-Gibbs measure}

In \cite{LNS} we introduced a measure on $[0, t]$ that plays a similar role to the Gibbs measure in the O'Connell-Yor polymer case, but is applicable for general potentials $V$. We recall in this section its definition, as well as some of its properties useful for our purposes.

For any bounded measurable function $F : \rr \to \rr$ supported in $[0, t]$, we define the pseudo-Gibbs expectation by,
\beq \label{eqn:psg-def}
\Eet_{N, t} [F] := \int_{0 < s_0 < \dots < s_{N-1} < t } \exp \left( - \sum_{j=0}^{N-1} \int_{s_j}^{s_{j+1} } V'' (u_j ( u) ) \d u \right) F(s_0)  \prod_{j=0}^{N-1} V'' (u_{j+1} (s_j ) ) \d s,
\eeq
where we abbreviated $u_j (t) = u_j ( t , \eta , \theta)$ above. The attribute ``pseudo" refers to the fact that, when $V(x)\neq \e^{-x}$, this measure is not associated in an obvious way to a natural polymer ensemble. (One could take \eqref{eqn:psg-def} to define a random ensemble of up-right paths in a random environment depending on the solutions $u_j$, but the significance of this interpretation is not clear to us. Additionally, the interpretation of the terms in the product on the right of \eqref{eqn:psg-def} as a weight of a polymer path in a traditional sense is unclear.) In any case, the main properties of this measure that we use are that it defines a non-negative measure on the simplex of jump times (so that we can use standard tools like Jensen's inequality), and do not use any real interpretation of it as a path ensemble.

From \cite[Proposition 6.1]{LNS} we have,
\bel
The assignment above defines a measure on the interval $[0, t]$ with total mass less than or equal to $1$.
\eel

In a similar fashion  to the proof of \cite[Lemma 7.2]{LNS} one can show that,
\beq\label{eqn:deriv-formula}
\del_\theta W_{N, t} ( \eta , \theta ) = \Eet_{N, t} [ ( s_0 )_+ ] ,
\eeq
linking the pseudo-Gibbs measure to the height function.

The following is \cite[Proposition 6.8]{LNS}. It gives a tail estimate for the random variable $s_0$ under the annealed pseudo-Gibbs measure. 
\bep
Let $I_0$ be a compact interval supported in $(0, \infty)$ on which $\psiV_2 ( \theta ) < 0$ for all $ \theta \in I_0$. Then there is a $\delta >0$ and $C_0 >0$ so that for all pairs $(t , \theta)$ satisfying $\theta \in I_0$ and
\[
|t-N \psiV_1 ( \theta ) | \leq \delta N
\]
we have for all $0 \leq w \leq \delta N$ that,
\[
\ee\left[ \Ett_{N, t} [ \1_{ \{ s_0 > \mff (\theta, t) + w \} } ] \right] \leq \exp \left( - \frac{ w^3}{16 N^2  \psiV_2 ( \theta)^2 } + C_0 N^{-3} w^4 \right)
\]
as long as $\mff ( \theta ,t ) + w \geq 0$.
\eep

From the above we easily conclude the following, in a similar fashion to the proof of Corollary \ref{cor:disc-tail-1}. 
\bec \label{cor:diff-jump}
Let $I_0$ and $\delta >0$ be as above. For all pairs $(t, \theta)$ satisfying $\theta \in I_0$ and $| \mff ( t , \theta ) | \leq \delta N$ we have for all $ a >0$ that,
\[
 \ee[ \Ett_{N, t} [ \e^{ a (s_0 )_+ } ] ] \leq C \e^{ C N a^{3/2}  + a \mff ( \theta , t )_+}.
\]
\eec

\section{Tails of diffusions} \label{sec:diff-tail}

The following proposition and corollary establishes upper and lower bounds for the moment generating function and tail of the height function $W_{N, t}^{\theta_0}$ in the special case of vanishing characteristic direction. We will later deduce the general case from this result and stationarity. 
\bep
Let $\theta_0$ satisfy $\psiV_2 ( \theta_0) < 0$.  Let $t_0 = N \psiV_1 (\theta_0)$. There is a $\delta >0$ so that for all $0 < a < \delta $ we have for some $c, C>0$ that
\beq \label{eqn:diff-lap-1}
c \e^{ c N a^{3} } \leq \ee \left[ \exp \left\{ a \left(  W_{N, t_0}^{\theta_0}  - \ee[ W_{N, t_0}^{\theta_0} ] \right) \right\} \right] \leq C \e^{ C N a^{3} }
\eeq
and
\beq \label{eqn:diff-lap-2}
\ee\left[ \exp \left\{ -a \left( W_{N, t_0}^{\theta_0}  - \ee[ W_{N, t_0}^{\theta_0} ] \right) \right\} \right] \leq C \e^{ C N a^3}
\eeq
\eep
\proof 
Let $\theta > \eta > 0 $. We have,
\[
W_{N, t} ( \theta , \theta ) = W_{N, t} ( \theta , \eta )  + \int_\eta^\theta ( \del_{u_2} W_{N, t} ) ( \theta, u_2 ) \vert_{u_2 = u } \d u. 
\]
By Proposition \ref{prop:diff-derivs} and \eqref{eqn:deriv-formula} we have,
\[
0 \leq \int_\eta^\theta ( \del_{u_2} W_{N, t} ) ( \theta, u_2 ) \vert_{u_2 = u }\d u \leq ( \theta - \eta ) \Ett_{N, t} [ (s_0)_+ ].
\]
In particular,
\beq \label{eqn:diff-lb}
\ee\left[ \exp ( ( \theta - \eta ) W_{N, t} ( \theta , \theta ) ) \right] \geq \ee\left[ \exp ( ( \theta - \eta ) W_{N, t} ( \theta , \eta ) ) \right]
\eeq
as well as,
\begin{align*}
&\ee\left[ \exp \left( \frac{1}{2}( \theta - \eta ) W_{N, t} ( \theta , \theta ) \right) \right]^2 \notag\\
\leq &~\ee\left[ \exp ( ( \theta - \eta ) W_{N, t} ( \theta , \eta ) ) \right] \ee\left[ \exp ( ( \theta - \eta)^2 \Ett[ (s_0 )_+ ]) \right] ,
\end{align*}
by Cauchy-Schwarz. 
By a variant of Jensen's inequality for sub-probability measures,
\[
\ee\left[ \exp ( ( \theta - \eta)^2 \Ett[ (s_0 )_+ ] \right] \leq 1 +\ee[ \Ett_{N, t} [ \e^{ (\theta - \eta)^2 (s_0 )_+ } ] ] .
\]
We specialize now to $\theta = \theta_0$, $\eta = \theta_0 -a$ and $t = t_0$. 
The parameters $t_0$ and $\theta_0$ in the statement of the Proposition are chosen so that $ \mff ( \theta_0 , t_0 ) = 0$. It follows then from Corollary \ref{cor:diff-jump} that,
\[
\ee \left[ E^{ ( \theta_0 \theta_0 ) }_{N, t_0} [ \e^{ a^2 (s_0 )_+ } ] \right]  \leq C \e^{ C Na^3}.
\]
Moreover, from Corollary \ref{cor:diff-generat} we have,
\[
\ee\left[ \exp ( ( \theta_0 - \eta ) W_{N, t_0} ( \theta_0 , \eta ) ) \right] \leq \exp\left( a \ee[ W_{N, t} ( \theta_0, \theta_0 ) ] + c N a^3 \right)
\]
by taking $\delta$ sufficiently small.  This proves the upper bound of \eqref{eqn:diff-lap-1}. The lower bound of \eqref{eqn:diff-lap-2} follows from \eqref{eqn:diff-lb} with $\theta = \theta_0$ and $\eta = \theta_0 - a$ and Corollary \ref{cor:diff-generat}.

For \eqref{eqn:diff-lap-2} we start with the inequality, for $\eta > \theta_0$,
\begin{align*}
W_{N, t_0} ( \theta_0 , \eta ) &= W_{N, t_0} ( \theta_0 , \theta_0 ) + \int_{\theta_0}^\eta ( \del_{u_2}  W_{N, t_0}  ) ( \theta_0  ,u_2 ) \vert_{u_2 = u} \,\mathrm{d}u \notag\\
&\leq  W_{N, t_0} ( \theta_0 , \theta_0 ) + ( \theta_0 - \eta ) E^{( \eta , \eta ) }_{N, t_0} [ (s_0 )_+],
\end{align*}
with the inequality following from Proposition \ref{prop:diff-derivs} and \eqref{eqn:deriv-formula}. We then have,
\begin{align*}
 & \ee\left[ \exp \left( \frac{ \theta_0 - \eta}{2} W_{N, t_0} ( \theta_0 , \theta_0) \right) \right]^2 \notag\\
\leq & \ee \left[ \exp \left( ( \theta_0 - \eta ) W_{N, t_0} ( \theta_0 , \eta) \right) \right] \ee \left[ \exp \left( (\theta_0 - \eta)^2 E^{ ( \eta , \eta ) }_{N, t_0} [ (s_0)_+ ] \right) \right]
\end{align*}
Now since $|\mff ( \eta , t_0) | \leq C N (\eta - \theta_0)$ we see that, arguing as above,
\beq
 \ee \left[ \exp \left( (\theta_0 - \eta)^2 E^{ ( \eta , \eta ) }_{N, t_0} [ (s_0)_+ ] \right) \right] \leq C \e^{ C N a^3 }.
\eeq
The claim now follows similarly to the argument above.
 \qed
 
 We deduce the following estimates on the tail from the estimates on the moment generating function via the routine Proposition \ref{prop:a-tail}.
 
 \bec \label{cor:tail-vanish}
 Let $\theta_0 >0$ be a point such that $\psiV_2 ( \theta_0 ) < 0$ and let $t_0 = N \psiV_2 ( \theta_0)$. There is a $\delta >0$ and $c, C>0$ so that for all $0 < u < \delta N$ we have,
 \[
 \pp\left[ |W_{N, \theta_0}^{t_0} - \mathbb{E}[W_{N, \theta_0}^{t_0}] | > u \right] \leq C \e^{ - c u^{3/2} N^{-1/2} }
 \]
 and
 \[
  \pp\left[ W_{N, \theta_0}^{t_0} - \mathbb{E}[W_{N, \theta_0}^{t_0}]  > u \right] > c \e^{- C u^{3/2} N^{-1/2} }.
 \]
 
 \eec

 \subsection{Proof of Theorem \ref{thm:main-diff}}
 
 Let $\theta_0 >0$ be as in the statement of the Theorem. Let $t_0 := N \psiV_1 ( \theta_0 )$. First assume that $ t > t_0$. Let
 \[
 W_{N, t}^{ \theta_0 } = \left( W_{N, t}^{ \theta_0 } + B_0 (t-t_0) - \theta ( t -t_0) \right) - B_{t-t_0} +\theta ( t -t_0) =: X + Y + \theta ( t -t_0).
 \]
 Then $X$ has the same distribution as $W_{N, t_0}^{\theta_0}$ and $Y$ is a Gaussian with variance $t-t_0 = | \mff ( \theta_0, t)|$. Therefore by Corollary \ref{cor:tail-vanish} we have
 \[
 \pp\left[ \left|( X - \ee[X] ) + Y \right| > u \right] \leq C \left( \e^{ - c u^{3/2} N^{-1/2} } + \e^{ - u^2 | \mff ( t , \theta ) |^{-1} } \right) .
 \]
For lower tail estimates we instead use,
\[
 \pp\left[ (X - \ee[X] ) + Y > u \right] \geq c  \e^{ - C u^{3/2} N^{-1/2} } -C \e^{ -c u^2 | \mff ( t , \theta ) |^{-1} }  .
\]
 The above estimates suffice to prove the theorem in the case $t > t_0$. The case $t < t_0$ is handled by interchanging the roles of $t$ and $t_0$ in the argument above. \qed

 \appendix
 \section{Tail estimates from moment generating functions}  \label{a:tail}
 \bep \label{prop:a-tail}
 Let $X$ be a random variable, and constants $c, C, \delta >0$ and $N \geq 1$ such that the estimates,
 \beq \label{eqn:a-1}
 c \e^{ c  N a^3} \leq \ee\left[ \e^{ a X} \right] \leq C \e^{ C N a^3}
 \eeq
 hold for all $a \leq \delta N$. Then there are constants $c', C'>0$ and $\delta' >0$ depending only on $c, C, \delta$ so that
 \beq \label{eqn:a-2}
 c' \e^{ - C' u^{3/2} N^{-1/2} } \leq \pp \left[ X > u \right] \leq C' \e^{ - c' u^{3/2} N^{-1/2} }
 \eeq
 for $0 < u < \delta' N$. If only the upper bound holds in \eqref{eqn:a-1} then the upper bound still holds in \eqref{eqn:a-2}.
 \eep
 \proof The upper bounds follows from Markov's inequality. For the lower bound, let $u_0 >0$. Then let $\delta > 2 a >0$ and $0 < u_0 < u_1$. Then,
 \begin{align*}
 c \e^{ c N a^3} \leq \ee[ \e^{a X} ] \leq \e^{ a u_1} \pp[ X > u_0] + \e^{ a u_0} + \ee[ \e^{2 a X} ]^{1/2} \pp[ X  > u_1 ]^{1/2}
 \end{align*}
 Choosing $a = C' (u_0 / N)^{1/2}$ for some large $C' >0$ (and assuming $u_0$ sufficiently small so that the requirement $2a < \delta$ is respected) we see that
 \[
 c \e^{ c N a^3} \leq \ee[ \e^{a X} ] \leq \e^{ a u_1} \pp[ X > u_0]+ \ee[ \e^{2 a X} ]^{1/2} \pp[ X  > u_1 ]^{1/2}
 \]
 for some new $c>0$.  By our upper bounds and choice of $a$,
 \[
 \ee[ \e^{2 a X} ]^{1/2} \pp[ X  > u_1 ]^{1/2} \leq C \e^{ C C'^3 u_0^{3/2} N^{-1/2}} \e^{ - c u_1^{3/2} N^{-1/2}}.
 \]
 Taking $u_1 = C'' u_0$ for some large $C'' >0$ we see that,
 \[
 c \e^{ c N a^3} \leq \e^{ a u_1} \pp [ X > u_0].
 \]
 This yields the claim. \qed


\begin{thebibliography}{}
\bibitem{BDMMZ} J. Baik, P. Deift, K. McLaughlin, P. Miller, and X. Zhou. \emph{Optimal tail estimates for directed last passage site percolation with geometric random variables.} Adv. Theor. Math. Phys., 5 no. 6, 2001.
\bibitem{B} M. Balazs, Private communication, 2022.
\bibitem{BC} G. Barraquand, I. Corwin, \emph{Random-walk in Beta-distributed random environment}, Probab. Theor. Related Fields, 167, 2017.
\bibitem{BCD} G. Barraquand, I. Corwin, E. Dimitrov. \emph{Maximal free energy of the log-gamma polymer.} Preprint, arXiv:2105.05283 (2021).
\bibitem{BG} R. Basu, S. Ganguyly, \emph{Connectiing eigenvalue rigidity with polymer geometry: diffusive transversal fluctuations under large deviation}. 
Preprint, arXiv:1902.09510, 2019.
\bibitem{BGK} R. Basu, S. Ganguly, M. Hegde, and M. Krishnapur. \emph{Lower deviations in $\beta$-ensembles and law of iterated logarithm in last passage percolation.} Israel J. Math. 242, no. 1, 2021.
\bibitem{BC} A. Borodin, I. Corwin. \emph{Macdonald processes.} Probab. Theor. Rel. Fields,
158:225–400, 2014.
\bibitem{BCF} A. Borodin, I. Corwin, P. Ferrari. \emph{Free energy fluctuations for directed polymers in random media in 1+ 1 dimension}. Comm.  Pure Appl. Math., 67(7), 2014.
\bibitem{BCFV} A. Borodin, I. Corwin, P. Ferrari, B. Veto. \emph{Height fluctuations for the stationary KPZ equation}. Math. Phys. Anal. Geom. 18(1), 2015.
\bibitem{C} I. Corwin, Private communication, 2021.
\bibitem{CC} M. Cafasso, T.  Claeys. \emph{A Riemann‐Hilbert Approach to the Lower Tail of the Kardar‐Parisi‐Zhang Equation.} Communications on Pure and Applied Mathematics 75, no. 3, 2022.
\bibitem{CN} H. Chaumont, C. Noack, \emph{Fluctuation exponents for stationary exactly solvable lattice polymer models via a Mellin transform framework}, ALEA, 15, 2018.
\bibitem{CG} I. Corwin, P. Ghosal, \emph{KPZ equation tails for general initial data}. Electron. J. Probab., 25, 2020.
\bibitem{CG2} I. Corwin, P. Ghosal. \emph{Lower tail of the KPZ equation.} Duke Math. J. 169, no. 7, 2020.
\bibitem{CH} I. Corwin, M. Hegde. \emph{The lower tail of $ q $-pushTASEP.} arXiv preprint arXiv:2212.06806, 2022.
\bibitem{CSS} I. Corwin, T. Sepp\"al\"ainen, H. Shen, \emph{The strict-weak lattice polymer}, J. Stat. Phys. 160, 2015.
\bibitem{DT} S. Das, L.-C. Tsai. \emph{Fractional moments of the stochastic heat equation.}  Ann. Inst. H. Poincar{\'e} Probab. Statist., vol. 57, no. 2, 2021.
\bibitem{EGO} E. Emrah, N. Georgiou, J. Ortmann. \emph{Coupling derivation of optimal-order central moment bounds in exponential last-passage percolation}. Preprint, arXiv:2204.06613, 2022.
\bibitem{EJS} E. Emrah, C. Janjigian, T. Sepp{\"a}l{\"a}inen. \emph{Right-tail moderate deviations in the exponential last-passage percolation}. Preprint,  arXiv:2004.04285, 2020.
\bibitem{EJS2} E. Emrah, C. Janjigian, T. Sepp{\"a}l{\"a}inen. \emph{Optimal-order exit point bounds in exponential last-passage percolation via the coupling technique}. Preprint, arXiv:2105.09402, 2021.
\bibitem{SG} N. Georgiou, T. Seppäläinen. \emph{Large deviation rate functions for the partition function in a log-gamma distributed random potential.} Ann. Probab., 41, no. 6, 2013.
\bibitem{IS} T. Imamura,  T. Sasamoto. \emph{Free energy distribution of the stationary O’Connell–Yor directed random polymer model.} J. Phys. A., 50, 28, 2017.
\bibitem{J}  K. Johansson. \emph{Shape fluctuations and random matrices.} Comm. Math. Phys., 209, no. 2, 2000.
\bibitem{LNS} B. Landon, C. Noack, P. Sosoe. \emph{KPZ-type fluctuation exponents for interacting diffusions in equilibrium}. Ann. Probab., to appear (2023).
\bibitem{LS} B. Landon, P. Sosoe. \emph{Tail bounds for the O'Connell-Yor polymer.} Preprint, arXiv:2209.12704, 2022.
\bibitem{MFSV} G. Moreno Flores, T. Sepp\"al\"ainen, B. Valk\'o, \emph{Fluctuation exponents for directed polymers in the intermediate disorder regime.} Electron. J. Probab. 19, 1-28.
\bibitem{NS} C. Noack, and P. Sosoe. \emph{Concentration for integrable directed polymer models.} Ann. Inst. H. Poincar{\'e} Probab. Statist., vol. 58, no. 1, pp. 34-64, 2022.
\bibitem{Oquantum} N.  O’Connell, \emph{Directed polymers and the quantum Toda lattice.} Ann. Probab. 40, no. 2, pp. 437-458, 2012.
\bibitem{OY} N. O'Connell, M. Yor. \emph{Brownian analogues of Burke's theorem} Stochastic Proc. Appl., 96, no. 2, 2001.
\bibitem{OO} N. O'Connell, J. Ortmann, \emph{Tracy-Widom asymptotics for a random polymer model with gamma-distributed weights}, Electron. J. Probab., 20, 2015.
\bibitem{R} E. Rains, \emph{A mean identity for longest increasing subsequence problems}, Preprint, arXiv:math/0004082, 2000.
\bibitem{S} T. Sepp\"ail\"ainen, \emph{Scaling for a one dimensional directed polymer with boundary conditions}, Ann. Probab., 40 no. 1, 2012.
\bibitem{S2} T. Sepp\"al\"ainen. \emph{Coupling the totally asymmetric simple exclusion process with a moving interface.} Markov Process. Related Fields, 4(4), 1998. I Brazilian School in Probability (Rio de Janeiro, 1997).
\bibitem{SV} T. Sepp\"al\"ainen, B. Valk\'o. \emph{Bounds for scaling exponents for a 1+1 dimensional directed polymer in a Brownian environment.}, ALEA Lat. Am. J. Probab. Math. Stat. 7, 2010.
\bibitem{TL} T. Thiery, P. LeDoussal, \emph{On integrable directed polymer models on the square lattice}. J. Phys. A., 48, no. 46, 2015.
\bibitem{Ts} L.-C. Tsai, \emph{Exact lower-tail large deviations of the KPZ equation.} Duke Math. J., no. 1 (2022): 1-44.
\bibitem{FSW} T. Weiss, P. Ferrari, and H. Spohn. \emph{Reflected Brownian motions in the KPZ universality class}. Springer International Publishing, 2017.
\bibitem{X} Y. Xie, \emph{Limit Distributions and Deviation Estimates of Random Walks in Dynamic Random Environments}, PhD Thesis, Purdue University, 2022.
\end{thebibliography}
\end{document}